\documentclass[a4paper]{amsart}

\usepackage[latin1]{inputenc}
\usepackage[T1]{fontenc}
\usepackage{amssymb,amsmath,amsfonts,amsthm,amscd,mathrsfs}
\usepackage[all]{xy}
\usepackage{color}
\usepackage[margin=3cm]{geometry}

\usepackage{dsfont, bm, color}

\usepackage[all]{xy}
\usepackage{tikz-cd}
\usepackage{enumerate}

\usepackage{hyperref}
\hypersetup{colorlinks,linkcolor=[rgb]{0,0,1},citecolor=[rgb]{1,0,0}}

\newtheorem{thm}{Theorem}[section]
\newtheorem{prop}[thm]{Proposition}
\newtheorem{lemme}[thm]{Lemma}
\newtheorem{cor}[thm]{Corollary}

\newtheorem{defipro}[thm]{Definition-Proposition}

\theoremstyle{definition}
\newtheorem{defi}[thm]{Definition}

\theoremstyle{remark}

\newtheorem{nota}[thm]{Notation}
\newtheorem{rmk}[thm]{Remark}
\newtheorem{ex}[thm]{Example}

\newtheorem{question}[thm]{Question}

\DeclareMathOperator{\Z}{\mathbb{Z}}

\DeclareMathOperator{\N}{\mathbb{N}}

\DeclareMathOperator{\codim}{Codim}

\DeclareMathOperator{\rk}{rk}

\DeclareMathOperator{\Fix}{Fix}
\DeclareMathOperator{\Supp}{Supp}

\DeclareMathOperator{\id}{id}

\DeclareMathOperator{\Sym}{Sym}

\DeclareMathOperator{\diag}{diag}
\DeclareMathOperator{\Sing}{Sing}

\DeclareMathOperator{\GL}{GL}
\DeclareMathOperator{\Q}{\mathbb{Q}}

\DeclareMathOperator{\C}{\mathbb{C}}

\DeclareMathOperator{\tr}{tr}
\DeclareMathOperator{\ch}{ch}
\DeclareMathOperator{\td}{td}
\DeclareMathOperator{\vol}{vol}

\DeclareMathOperator{\orb}{orb}
\DeclareMathOperator{\reg}{reg}
\DeclareMathOperator{\SU}{SU}
\DeclareMathOperator{\Sp}{Sp}
\DeclareMathOperator{\U}{U}

\DeclareMathOperator{\SL}{SL}

\newcommand{\eq}[1][r]
{\ar@<-3pt>@{-}[#1]
\ar@<-1pt>@{}[#1]|<{}="gauche"
\ar@<+0pt>@{}[#1]|-{}="milieu"
\ar@<+1pt>@{}[#1]|>{}="droite"
\ar@/^2pt/@{-}"gauche";"milieu"
\ar@/_2pt/@{-}"milieu";"droite"}

\newcommand{\incl}[1][r]
  {\ar@<-0.2pc>@{^(-}[#1] \ar@<+0.2pc>@{-}[#1]}

\begin{document}
\title[Betti numbers of 4-dimensional symplectic orbifolds]{On the Betti numbers of compact holomorphic symplectic orbifolds of dimension four}

\author{Lie \textsc{Fu}}

\author{Gr\'egoire \textsc{Menet}} 
	
	\thanks{L. Fu is supported by the Radboud Excellence Initiative programme. G. Menet was financed by the Fapesp grant 2014/05733-9, the Marco Brunella grant of Burgundy University and the ERC-ALKAGE, grant No. 670846. }

	\date{\today}

\begin{abstract}
We extend a result of Guan by showing that the second Betti number of a 4-dimensional primitively symplectic orbifold is at most 23 and there are at most 91 singular points. The maximal possibility 23 can only occur in the smooth case. In addition to the known smooth examples with second Betti numbers 7 and 23, we provide examples of such orbifolds with second Betti numbers 3, 5, 6, 8, 9, 10, 11, 14 and 16. In an appendix, we extend Salamon's relation among Betti/Hodge numbers of symplectic manifolds to symplectic orbifolds.  
\end{abstract}

\maketitle

\section{Introduction}
A compact K\"ahler manifold is called \emph{holomorphic symplectic} if it admits a holomorphic 2-form that is nowhere degenerate. In particular, it is even-dimensional and has trivial canonical bundle. Such a manifold is called \emph{irreducible} if moreover it is simply connected and the holomorphic symplectic form is unique up to scalar. Irreducible holomorphic symplectic (IHS) manifolds (also known as compact \emph{hyper-K\"ahler} manifolds) admit a Ricci-flat Riemannian metric \cite{Yau}, and are characterized by the condition that the holonomy group is the compact symplectic group. The importance of IHS manifolds is manifested in the Beauville--Bogomolov decomposition theorem \cite{Beauville, Bogomolov}: any compact K\"ahler manifold with vanishing first Chern class has a finite \'etale cover which can be written as a product of a complex torus, Calabi--Yau varieties and IHS manifolds. We refer to \cite{Beauville}, \cite{HuyInv} and \cite[Part III]{GHJ03} for the basic theory on such manifolds.

Irreducible holomorphic symplectic surfaces are nothing but K3 surfaces. The construction problem for IHS manifolds in higher dimensions seems quite hard: up to deformation, in each even dimension ($\geq 4$), we so far only have two examples constructed by Beauville \cite{Beauville} (Hilbert schemes of points on K3 surfaces and generalized Kummer varieties) together with two sporadic examples constructed by O'Grady \cite{O'G99, O'G03} in dimensions 6 and 10. The limitedness of available examples suggests the possibility to bound or even classify IHS manifolds (see \cite{Huy03} for diffeomorphic types). As the second cohomology of an IHS manifold, together with the Beauville--Bogomolov quadratic form \cite{Beauville} and the weight-2 Hodge structure, controls most of its geometry \cite{VerbitskyTorelli, Mar11, Huy12, Bakker, Lol, BakkerLehn, VerbitskyTorelliErratum},  it is natural to ask the following question.

\begin{question}\label{ques:b2}
In a given dimension, what values can the second Betti number of an irreducible holomorphic symplectic manifold take? 
\end{question}

In dimension 4, Guan \cite{Guan} proved the following result in the direction of Question \ref{ques:b2}.
\begin{thm}[Guan]\label{thm:Guan}
The second Betti number of a 4-dimensional irreducible holomorphic symplectic manifold is no more than 8, or equal to 23. Moreover, if $b_{2}=23$, the Hodge diamond must be the same as that of the Hilbert square of a K3 surface. 
\end{thm}
The fact that $b_{2}\leq 23$ was attributed to Beauville. See \cite{Guan} and \cite{Guan2} for extra constraints on each cases; see \cite{Kurnosov, Sawon15} for related results in dimension 6 and the more recent work \cite{KimLaza} for a conjectural bound in arbitrary dimension based on \cite{GKLR19}. When $b_{2}=23$, let us mention the work \cite{O'G08, Kapu16}, which aims at determining the deformation type of IHS fourfolds upon fixing some extra topological data.

\vspace{10pt}
In the point of view of birational geometry, or more precisely the minimal model program (\emph{cf}.~\cite{KollarMori}), it is important to treat varieties with mild singularities. With recent intensive efforts \cite{GKP16, Dru18, GGK, DG18, HP19}, the Beauville--Bogomolov decomposition theorem is now extended to projective varieties with klt singularities and numerically trivial canonical class. Naturally, boundedness results for possibly singular irreducible holomorphic symplectic varieties (\cite[Definition 8.16]{GKP16}) are desired. In particular, Question \ref{ques:b2} can be posed in this broader setting. 

This article is our first experimental attempt towards the boundedness problem,  where we will focus on the classical approach of enlarging the category of IHS manifolds to the so-called \emph{primitively symplectic orbifolds}, pioneered by Fujiki \cite{Fujiki}. Roughly speaking, a primitively symplectic orbifold is a compact K\"ahler space with quotient singularities in codimension $\geq 4$, such that the smooth locus carries a holomorphic symplectic form which is unique up to scalar. See Definition~\ref{def:PSO}. Our first main result extends Guan's Theorem \ref{thm:Guan}. 
\begin{thm}\label{main}
Let $X$ be a primitively symplectic orbifold of dimension 4. Then
\begin{equation*}
b_2(X)\leq 23.
\end{equation*}
Moreover, the equality occurs only in the smooth case.
\end{thm}

Our second main result bounds the size of the singularities.
\begin{thm}\label{main2}
Let $X$ be a primitively symplectic orbifold of dimension 4. Then
\begin{enumerate}[$(i)$]
\item $X$ has at most 91 singular points.
\item For each singular point of $X$, the order of the local fundamental group is at most $1424$.
\end{enumerate}
\end{thm}
The proof of Theorem \ref{main} and Theorem \ref{main2} is given in the end of Section \ref{sec:Proof}.\\

The bound for $b_{2}$ in Theorem \ref{main} being the same as in the smooth case 
(note however that no numbers between 9 and 22 are excluded as in \cite{Guan}, 
despite of our effort in Section \ref{sec:HSformula} where we generalize the Hitchin--Sawon formula), 
the construction methods in the orbifold setting are much richer. 
Indeed, staying in the smooth category of IHS fourfolds, the only available values for $b_{2}$ are 23 and 7, achieved by Hilbert squares of K3 surfaces and generalized Kummer fourfolds respectively; while we are able to construct much more examples within the enlarged category of symplectic orbifolds, filling many ``gaps'' in the possibilities of the Betti number. More precisely, we have the following result.  
\begin{thm}
There are 4-dimensional primitively symplectic orbifolds with second Betti number 3, 5, 6, 7, 8, 9, 10, 11, 14, 16 and 23. 
\end{thm}
We refer to Section \ref{sec:Examples} for details of these examples. 
~\\

\noindent\textbf{Acknowledgements.} 
We would like to thank Roland Bacher, Arnaud Beauville, Daniel Huybrechts and Radu Laza for useful discussions. We also want to thank the referee for his or her careful reading and helpful comments.
We are very grateful to the 
Second Japanese-European Symposium on Symplectic Varieties and Moduli Spaces where our collaboration 
was initiated.

\section{Riemann--Roch theorem for orbifolds}\label{reminders}
\subsection{Orbifolds and V-bundles}
We first fix the notion of orbifolds (``V-manifolds'' in \cite{Satake56, Satake57}).
\begin{defi}[Orbifolds]\label{def:orbifolds}
An \emph{$n$-dimensional complex orbifold} is a connected paracompact Hausdorff complex analytic space $X$ such that for any point $x\in X$, there exists an open neighborhood $U$ of $x$ and a triple $(V,G,\pi)$ with $V$ an open subset of $\C^n$, $G$ a finite \emph{subgroup} of the biholomorphic automorphism group of $V$, and $\pi:V\rightarrow U$ the composition of the quotient map $V\rightarrow V/G$ and an isomorphism $V/G\simeq U$. 
\end{defi}
\begin{rmk}\label{normal}
Note that an orbifold is always normal (see for instance \cite[Th\'eor\`eme 4]{Cartan}). In particular, the singular locus is of codimension at least 2.
\end{rmk}

\begin{defipro}[{\cite[Proposition 6]{Prill}}]
Let $X$ be an $n$-dimensional complex orbifold. Let $x\in \Sing X$. 
Then there exist a finite subgroup $G_x$ of $\GL_{n}(\C)$ and an open neighbourhood $V_x\subset \C^n$ of the origin $0\in \C^n$, stable under the action of $G_x$, with $V_x/G_x$ isomorphic to an open neighbourhood $U_x$ of $x$, and such that
$$\codim\Fix(g)\geq 2\ \text{for all}\ g\in G_x\backslash \{\id\}.$$
Such a group $G_x$ is unique up to conjugation. Let $\pi_x:V_x\rightarrow V_x/G_x\simeq U_x$, we call $(U_x,V_x,G_x,\pi_x)$ a \emph{local uniformizing system} of $x$, and $G_x$ the \emph{local fundamental group} of $X$ at $x$.
\end{defipro}

\begin{rmk}
	In modern literature (see for example \cite{ALRBook}), \emph{orbifold} is a synonym for (analytic) \emph{Deligne--Mumford stack}. In particular, the collection of charts $(V, G, \pi)$ is part of the data of an orbifold and the group $G$ is sometimes allowed to act non-effectively on $V$. If we use the terminology \emph{orbifold} in this generalized sense, then what is defined in Definition \ref{def:orbifolds} corresponds to \emph{effective orbifolds} \cite[Definition 1.2]{ALRBook} with all nontrivial elements of all stabilizer groups acting with a non-empty fixed locus of codimension at least 2. By the above result of Prill, these are equivalent notions, as such an orbifold/stacky structure is determined by the underlying complex analytic space. 
\end{rmk}

The notion of vector bundles naturally generalizes to orbifolds.

\begin{defi}[V-bundles]\label{def:VB}
Let $X$ be an orbifold. 
\begin{itemize}
\item
A V-\emph{bundle} (or \emph{orbibundle}) on $X$ is a vector bundle $F$ on $X_{\operatorname{reg}}:=X\smallsetminus\Sing X$ such that for any local uniformizing system $(U,V,G,\pi)$, there exists a vector bundle $\hat{F}_V$ on $V$ endowed with an equivariant action of $G$ such that: 
$$\hat{F}_{V|V\smallsetminus \Fix G}\simeq \pi^*(F_{|U_{\operatorname{reg}}}),$$
where $\Fix G:=\bigcup_{g\in G, g\neq \id} \Fix(g)$.
\item
Let $\mathscr{F}$ be a coherent sheaf on $X$. 
The sheaf $\mathscr{F}$ is said to be \emph{locally} V-\emph{free} if for any $x\in X$, there exist a local uniformizing system $(U,V,G,\pi)$, a free coherent sheaf $\hat{\mathscr{F}}_V$ on $V$, and a $G$-action on $\hat{\mathscr{F}}_V$ such that $\mathscr{F}_{|U}\simeq \pi_*\left(\hat{\mathscr{F}}_V^G\right)$. By \cite[4.2]{Blache}, the local V-freeness of a coherent sheaf $\mathscr{F}$ is equivalent to the condition that $\mathscr{F}$ is reflexive and the reflexive pull-back $\pi^{[*]}(\mathscr{F}_{|U}):=\pi^{*}(\mathscr{F}_{|U})^{\vee \vee}$ is locally free for any local uniformizing system.
\end{itemize}
As in the smooth case, there is an equivalence of categories between the category of locally V-free sheaves and that of V-bundles.
\end{defi}

\begin{ex}[Reflexive differentials]
Given an orbifold $X$ of dimension $n$, the sheaf of reflexive differential forms (\cite[2.D]{GKKP}, \cite[Section 2.5]{Peters}) $$\Omega_X^{[i]}:=(\Omega_{X}^{i})^{\vee \vee}\cong\iota_{*}(\Omega^{i}_{X_{\operatorname{reg}}})$$  is a locally V-free sheaf for any $i\in \mathbb{N}$, where $\iota:X_{\operatorname{reg}} \to X$ is the natural inclusion of the smooth part. The sheaf of reflexive forms of top degree is identified with the dualizing sheaf: $\omega_{X}\cong \Omega_{X}^{[n]}$.
\end{ex}

\begin{rmk}[Hodge decomposition]
Let $X$ be a compact K\"ahler orbifold.  For any integer $k\geq 0$, the rational singular cohomology group $H^{k}(X, \Q)$ carries a pure Hodge structure of weight $k$ and in the Hodge decomposition 
$$H^{k}(X, \C)=\bigoplus_{p+q=k} H^{p,q}(X),$$
we have $H^{p, q}(X)\cong H^{q}(X, \Omega_{X}^{[p]})$, see \cite[Section 2.5]{Peters}. We denote $h^{p,q}(X):=\dim H^{p, q}(X)$. We have $h^{p,q}=h^{q,p}$.
\end{rmk}


\begin{nota}\label{repre}
Let $X$ be an orbifold, $x\in X$ and $\mathscr{F}$ a locally V-free sheaf. Let $(U,V,G,\pi)$ be a local uniformizing system of $x$ and let $\hat{\mathscr{F}}_V$ be a locally free sheaf on $V$ endowed with an action of $G$ as in Definition \ref{def:VB}. Hence, the fiber of $\hat{\mathscr{F}}_{V}$ at $0$ is endowed with an action of $G$, which provides a representation of $G$. 
We denote by $\rho_{x,\mathscr{F}}$ \emph{the representation of $G$ associated with $x$ and $\mathscr{F}$}.
\end{nota}

\subsection{Characteristic classes on orbifolds}
We recall the definition of Chern classes of V-bundles on orbifolds, by adapting the Chern--Weil approach.
\begin{defi}[{\cite[Definition 2.9]{Blache}}]\label{metric}
 Let $F$ be a V-bundle of rank $r$ on an orbifold $X$.
A \emph{metric} on $F$ is a Hermitian metric $h$ on $F$ as bundle on $X_{\reg}$ such that for all local uniformizing systems $(U,V,G,\pi)$, the Hermitian metric $\pi^*(h_{|U_{\reg}})$ extends to a Hermitian metric on $\hat{F}_V$. 
\end{defi}

\begin{defi}[{\cite[Definition 1.5]{Blache}}]\label{form}
 Let $X$ be an orbifold. A \emph{smooth differential $k$-form} $\varphi$ on $X$ is a $C^{\infty}$ differential $k$-form on
 $X_{\reg}$ such that for all local uniformazing system $(U,V,G,\pi)$, $\pi^*(\varphi_{|U_{\reg}})$ extends to a
 $C^{\infty}$-differential $k$-form on $V$. (We always mean $\C$-valued forms.)
\end{defi}
\begin{nota}
 We denote by $\mathcal{A}^{k}$ the sheaf of differential $k$-forms.
\end{nota}

As explained in \cite[Definition 2.10]{Blache}, we can define the \emph{Chern classes} of a V-bundle as follows. 
Let $F$ be a V-bundle of rank $r$ on an orbifold $X$.
We can first construct the Chern forms on $X_{\reg}$ as in the smooth case. To a Hermitian metric $h$ on $F$, we associate the Chern connection $D$ on $F$, and to $D$, we associate the curvature $D^2$. 
Let $\Xi$ be the corresponding $r\times r$ matrix of curvature 2-forms, then we set $c_k(h)=P_k(\frac{i}{2\pi}\Xi)\in \Gamma(X_{\reg},\mathcal{A}^{2k})$, where $P_k$ is the $k$-th elementary invariant polynomial 
function $\C^{r\times r}\rightarrow \C$. 

The same process can be also done on all local uniformizing systems $(U,V,G,\pi)$. The metric $\pi^*(h_{|U_{\reg}})$
extends to a Hermitian metric $\hat{h}$ on $\hat{F}_V$ which gives rise to the Chern connection $\hat{D}$ on $\hat{F}_V$ and hence the curvature $\hat{D}^2$. 
As previously, we can construct $c_k(\hat{h})\in \Gamma(V,\mathcal{A}_V^{2k})$.
By construction $\pi^*(c_k(h)_{|U_{\reg}})$ extends to $c_k(\hat{h})$ on $V$.
Hence, we obtain $c_k(h)\in \Gamma(X,\mathcal{A}^{2k})$. 
As in the smooth case, we show that $c_k(h)$ is a closed form, and that the cohomology class
$c_k(F):=\left[c_k(h)\right]\in H^{2k}(X,\C)$, called the $k$-th Chern class of $F$,
 does not depend on the choice of the metric $h$. 

Other characteristic classes, like Todd classes and Chern characters, are defined in terms of Chern classes by the usual formulas. A characteristic class of an orbifold is that of its tangent V-bundle.

\subsection{Riemann--Roch and Gauss--Bonnet theorems for orbifolds}\label{subsec:RRGB}
One key ingredient in the proof of Theorem \ref{main} is the following orbifold version of the Hirzebruch--Riemann--Roch theorem due to Blache \cite[Theorems 3.5 and 3.17]{Blache}.
\begin{thm}[Blache {\cite{Blache}}]\label{thm:RR}
Let $X$ be a compact complex orbifold with only isolated singularities and let $\mathscr{F}$ be a locally {\rm V}-free sheaf.
Then we have
$$\chi(X, \mathscr{F})=\int_X \ch(\mathscr{F})\cdot\td(X)+\sum_{x\in\Sing X}\left[\frac{1}{|G_x|}\sum_{\substack{g\in G_x\\ g\neq \id}} \frac{\tr(\rho_{x,\mathscr{F}}(g))}{\det(\id-\rho_{x, T_{X}}(g))}\right],$$
where $g$ is viewed as an endomorphism on $T_0V$ with $(U,V,G_x,\pi)$ a local uniformizing system of $x$.
\end{thm}

Blache also established the following orbifold version of Gauss--Bonnet theorem.
\begin{thm}[{\cite[Theorem 2.14]{Blache}}]\label{cara}
Let $X$ be an $n$-dimensional compact complex orbifold with only isolated singularities. Then we have the following formula for its topological Euler characteristic:
$$\chi_{\operatorname{top}}(X)=\int_{X} c_n(X)+\sum_{x\in\Sing X}\left(1-\frac{1}{|\"aG_x|}\right),$$
where $G_x$ is the local fundamental group of $X$ at $x$.
\end{thm}

\section{Bounding Betti numbers and singularities}\label{sec:Proof}
The aim of this section is to show Theorem \ref{main} and Theorem \ref{main2}. Let us first make precise the class of (possibly singular) symplectic varieties that we consider.

\subsection{Symplectic orbifolds} 
\begin{defi}[Fujiki \cite{Fujiki}]\label{def:PSO}
A compact K\"ahler orbifold $X$ is called \emph{primitively symplectic} if
\begin{enumerate}[$(i)$]
\item
the smooth locus $X_{\operatorname{reg}}:=X\smallsetminus \Sing X$ is endowed with a non-degenerated holomorphic 2-form which is unique up to scalar; and
\item
the singular locus $\Sing X$ has codimension at least 4.
\end{enumerate}
If moreover $X_{\operatorname{reg}}$ is simply connected, $X$ is called an \emph{irreducible symplectic orbifold}. \end{defi}

\begin{rmk}
As in the smooth case, a primitively symplectic orbifold $X$ has even (complex) dimension and trivial dualizing sheaf $\omega_{X}\simeq \mathcal{O}_{X}$. Moreover, the symplectic form extends to a symplectic form on any local uniformizing system. In particular, the contraction with the symplectic form induces an isomorphism $T_{X}\simeq \Omega_{X}^{[1]}$. By definition, if $X$ has dimension 4, then $X$ has isolated quotient singularities. 
\end{rmk}

\begin{rmk}
As quotient singularities are rational singularities, the singularities appearing in Definition \ref{def:PSO} are symplectic singularities in the sense of Beauville \cite{BeauvilleSympSing}.  Moreover, an irreducible symplectic orbifold defined above is an irreducible symplectic variety in the sense of \cite{GKP16} and \cite[Definition 1.4]{HP19}.
\end{rmk}
\begin{rmk}[Hodge diamond]\label{b1}
Let $X$ be a 4-dimensional primitively symplectic orbifold. Fujiki {\cite[Proposition 6.7]{Fujiki}}  showed that $X$ has vanishing irregularity, hence $b_{1}(X)=0$. Serre--Grothendieck duality gives that $$H^{3}(X, \Omega_{X}^{[1]})\cong H^{1}(X, \omega_{X}\otimes T_{X})^{\vee}\cong H^{1}(X, \Omega_{X}^{[1]})^{\vee}.$$ In particular, $h^{3,1}=h^{1,1}$. Similarly, $h^{3,0}=h^{1,0}=0$. In conclusion, the Hodge diamond of $X$ takes the following form.
\begin{center}
	\begin{tabular}{ccccccccc}
	&&&&1&&&&\\
	&&&0&&0&&&\\
	&&1&&$h^{1,1}$&&1&&\\
	&0&&$h^{2,1}$&&$h^{2,1}$&&0&\\
	1&&$h^{1,1}$&&$h^{2,2}$&&$h^{1,1}$&&1.\\
	&0&&$h^{2,1}$&&$h^{2,1}$&&0&\\
	&&1&&$h^{1,1}$&&1&&\\
	&&&0&&0&&&\\
	&&&&1&&&&
	\end{tabular}
	\end{center}	
\end{rmk}

\subsection{Quotient symplectic singularities in dimension 4}
For later use, we classify in this section all symplectic quotient singularities 
in dimension 4. As the germ of a quotient singularity is determined by the local fundamental group, one needs to classify all finite subgroups of the Lie group $\Sp(4,\C)$. Since any finite subgroup must be contained in some compact maximal subgroup, we are to classify finite subgroups of the \emph{compact} symplectic group 
$\Sp(2):=\Sp(4,\C)\cap \SU(4)$.
\begin{prop}\label{classification}
Let $n>0$ be an integer, we denote $\xi_n:=e^{\frac{2i\pi}{n}}$ the primitive $n$-th root of unity.
For integers $1\leq k\leq n$, we denote
$$T_{n,k}:=\begin{pmatrix}
      0 & 0 & 1 & 0\\
      0 & 0 & 0 & 1\\
      \xi_n^k & 0 & 0 & 0\\
      0 & \xi_n^{-k} & 0 & 0
     \end{pmatrix}.$$
Let $G$ be a finite subgroup of the compact symplectic group $\Sp(2)$. 
Then, up to conjugation,
\begin{itemize}
\item[$(i)$]
there exists finite subgroups $H_1$, $H_2$ of $\SU(2)$, 
 integers $n>0$ and $k\in \{1,...,n\}$, and a normal subgroup
$G'$ of $G$ of index at most 2, such that any element $M'$ of $ G'$ has the form
$$M'=\begin{pmatrix}A & 0\\ 0 & B
      \end{pmatrix},$$
with $A\in H_1$, $B\in H_2$, and 
$G/G'=\left<\overline{T_{n,k}}\right>$  if $G'\neq G$.
\item[$(ii)$]
If moreover $\C^4/G$ has only the image of $0$ as singularity, then
there exists  a finite subgroup $H$ of $\SU(2)$ and $\theta$ an automorphism of $H$
such that any element $M\in G'$ has the form
$$M=\begin{pmatrix}A & 0\\ 0 & \theta(A)
      \end{pmatrix},$$
for some $A\in H$.
\end{itemize}
\end{prop}
\begin{proof}
Hanany and He classified in \cite{Hanany} the finite subgroups of $\SU(4)$. 
Hence it is enough to identify those groups in the list that preserve a symplectic form. In the sequel, we follow their notation. 

The first category of groups are the so-called \emph{primitive simple groups} described in \cite[Section 3.1.1]{Hanany}
and they are numbered from I to VI. However, none of them are symplectic. Indeed, the following matrices are considered:
$$F_1:=\begin{pmatrix}1 & 0 & 0 & 0\\ 0 & 1 & 0 & 0\\ 0 & 0 & w & 0\\ 0 & 0 & 0 & w^2 \end{pmatrix},\ 
F_2:=\frac{1}{\sqrt{3}}\begin{pmatrix}1 & 0 & 0 & \sqrt{2}\\0 & -1 & \sqrt{2} & 0 \\ 0 & \sqrt{2} & 1 & 0\\ \sqrt{2} & 0 & 0 & -1
\end{pmatrix}\ \text{and}\ F_2':=\frac{1}{3}\begin{pmatrix}3 & 0 & 0 & 0\\0 & -1 & 2 & 2 \\ 0 & 2 & -1 & 2\\ 0 & 2 & 2 & -1
\end{pmatrix},$$
where $w:=e^{\frac{2i\pi}{3}}$.
The matrices $F_1$ and $F_2$ do not preserve any common symplectic form, hence the groups I and III, 
which are partially generated by these two matrices, cannot be symplectic.
Similarly, the group II cannot be symplectic because it is partially
generated by the two matrices $F_1$ and $F_2'$ which do not fix the same symplectic form.

Let $\beta:=e^{\frac{2i\pi}{7}}$.
The matrices $S:=\diag(1,\beta, \beta^4, \beta^2)$ and $D:=\diag(w,w,w,1)$ 
are not symplectic, hence the groups IV, V and VI, which are partially generated by one of these two matrices, are not symplectic. 

In \cite[Section 3.1.2]{Hanany}, Hanany and He consider the groups VII, VIII and IX which cannot be
symplectic since they are partially generated by the groups I, II and III.

In \cite[Section 3.1.3]{Hanany}, they consider group obtained from Kronecker products of matrices of $\SU(2)$.
Let $$S_{\SU(2)}:=\frac{1}{2}\begin{pmatrix}-1+i& -1+i\\ 1+i & -1-i\end{pmatrix}\ 
\text{and}\ U_{\SU(2)}:=\frac{1}{\sqrt{2}}\begin{pmatrix}1+i& 0\\ 0 & 1-i\end{pmatrix}.$$
The following couples of matrices $(S_{\SU(2)}\otimes S_{\SU(2)},U_{\SU(2)}\otimes U_{\SU(2)})$ and 
$(S_{\SU(2)}\otimes S_{\SU(2)},U_{\SU(2)}^2\otimes U_{\SU(2)}^2)$ both do not fix the same symplectic form.
Hence the groups from X to XVI cannot be symplectic 
since they are all partially generated by one of theses couples of matrices.
Also the groups from XVII to XXI cannot be symplectic because they are partially generated by the groups XI, X, XVI and XIV.

The matrices $A_1:=\diag(1,1,-1,-1)$ and $A_2:=\diag(1,-1,-1,1)$ do not fix the same symplectic form. 
Hence all the groups from XXII to XXX are not symplectic since they are all partially generated by these two matrices.

Finally, we consider the group:
$$\Delta:=\left\{\left.\diag(w^j, w^k,  w^{l}, w^{-j-k-l})\ \right|\ w=e^{\frac{2i\pi}{n}},\ j, k, l\in\{1,...,n\}\right\},$$
which is not symplectic, hence all the groups from XXXI to XXXIII which are partially generated by $\Delta$ are not symplectic.

It only remains to consider the group XXXIV and the intransitive groups. We will study the group XXXIV in the end. 
The intransitive groups are the groups coming from diagonal embedding of subgroups of $\SU(2)$ or $\SU(3)$
(see \cite[Definition 2.1]{Hanany} for the precise definition).
We will see that all the symplectic groups constructed from a diagonal embedding of a subgroup of $\SU(3)$ are actually constructed from a
diagonal embedding of subgroups of $\SU(2)$.
Let $G$ be such a group and $M$ an element in $G$. Let $(e_1, e_2, e_3, e_4)$ be the canonical basis of 
$\C^4$. We have: 
$$M=\begin{pmatrix}
   \xi & 0\\
   0 & A
  \end{pmatrix},$$
  where $\xi$ is a root of the unity and $A\in \U(3)$.
We can find a basis $(v_1,v_2,v_3)$ of $\C^3$ in which $A$ is diagonalized: $A=\diag(\xi^{-1},\zeta,\zeta^{-1})$.
In the basis $(e_1,v_1,v_2,v_3)$, the symplectic form has to be 
$$J:=\begin{pmatrix}
                  0 & 1 & 0 & 0\\
                  -1 & 0 & 0 & 0\\
                  0 & 0 & 0 & 1\\
                  0 & 0 & -1 & 0
                 \end{pmatrix}.$$
We consider now another matrix of $G$ expressed in the basis $(e_1,v_1,v_2,v_3)$:                 
$$N:=\begin{pmatrix}
                  a & 0 & 0 & 0\\
                  0 & b & c & d\\
                  0 & e & f & g\\
                  0 & h & j & k
                 \end{pmatrix},$$
where $a,b,c,d,e,f,g,h,j,k$ are in $\C$. 
Since $N$ is symplectic. It follows:
\begin{align*}
 ab&=1\\
 ac&= 0\\
 ad&=0\\
 -hf+ej&=0\\
 -hg+ek&=0\\
 -jg+fk&=1.
\end{align*}
Hence $c=d=0$. If $h\neq 0$, then $f=\frac{ej}{h}$ and $g=\frac{ek}{h}$.
This is impossible because, it contradicts $-jg+fk=1$. So $h=0$. 
For the same reason $e=0$ and $G$ is actually a group composed by matrices of the forms: 
$$N:=\begin{pmatrix}
                  a & 0 & 0 & 0\\
                  0 & b & 0 & 0\\
                  0 & 0 & f & g\\
                  0 & 0 & j & k
                 \end{pmatrix}.$$
It only remains to study the case of the group XXXIV.
In this case, $G=\left<G_0,T_{n,k}\right>$, with $G_0$ composed of matrices of type
$$\begin{pmatrix}A& 0\\ 0 & B \end{pmatrix},\text{ with }\ A,B\in \SU(2).$$
We consider 
$$G':=\left\{M=\begin{pmatrix}A& 0\\ 0 & B \end{pmatrix},\in G\left |\ A,B\in \SU(2)\right.\right\}.$$
The group $G'$ is a normal subgroup of $G$ and the class $\overline{T_{n,k}}\in G/G'$ has order 2.

Now, we prove $(ii)$. 
Let $G$ be a finite subgroup of $\Sp(2)$ such that $\C^4/G$ admits only 0 as singularities.
Then necessarily, the unique element of $G$ with the eigenvalue 1 is the identity. In particular, this is true for $G'$.
Therefore, the following projections are isomorphisms:
$$P_1:G'\longrightarrow H_1,\ \begin{pmatrix}A& 0\\ 0 & B \end{pmatrix}\mapsto A,\ \text{and}\ 
P_2:G'\longrightarrow H_2,\ \begin{pmatrix}A& 0\\ 0 & B \end{pmatrix}\mapsto B.$$
So, setting $\theta:=P_2\circ P_1^{-1}$ finishes the proof.
\end{proof}

\subsection{Orbifold Salamon relation}\label{secRR}
Salamon \cite[(0.1)]{Salamon} discovered a remarkable linear relation among the Betti/Hodge numbers of a compact hyper-K\"ahler manifold. By applying Blache's orbifold Hirzebruch--Riemann--Roch Theorem \ref{thm:RR}, we will establish (Proposition \ref{RR} below) a Salamon-type relation between the Hodge numbers and information from the singularities of a 4-dimensional primitively symplectic orbifold. A more general result can be obtained by adapting Salamon's method. However, we decide to give an elementary proof here and leave the general result in Appendix \ref{App:Salamon}.

 In what follows,  $X$ is a primitively symplectic orbifold of dimension 4. For any (necessarily isolated) singular point $x\in X$, $G_x$ is the local fundamental group of $X$ at $x$ and $\rho_{x,\bullet}$ is the representation of $G_x$ defined in Notation \ref{repre}. 
\begin{itemize}
\item Define for any integer $p\geq 0$,
\begin{equation}\label{eqn:Si}
S_p:=\sum_{x\in\Sing X} \frac{1}{|G_x|}\sum_{g\neq\id}\frac{\tr(\rho_{x,\Omega_X^{[p]}}(g))}{\det(\id-\rho_{x, T_{X}}(g))}.
\end{equation}
\item Applying Theorem \ref{thm:RR} to $\mathscr{F}=\mathcal{O}_X$, we obtain:
\begin{equation}
\int_{X}\td_4(X)=3-S_0,
\label{0}
\end{equation}
\item Applying Theorem \ref{thm:RR} to $\mathscr{F}=\Omega_X^{[1]}$, we obtain, using Remark \ref{b1}:
\begin{equation}
h^{2,1}(X)-2h^{1,1}(X)=4\int_{X}\td_4(X)-\frac{1}{6}\int_{X}c_4(X)+S_1,
\label{1}
\end{equation}
\item Applying Theorem \ref{thm:RR} to $\mathscr{F}=\Omega_X^{[2]}$, we obtain, using Remark \ref{b1}:
\begin{equation}
2-2h^{2,1}(X)+h^{2,2}(X)=6\int_{X}\td_4(X)+\frac{2}{3}\int_{X}c_4(X)+S_2,
\label{2}
\end{equation}
\item Applying Theorem \ref{cara}, we obtain:
\begin{equation}\label{caraequa}
8+4h^{1,1}-4h^{2,1}+h^{2,2}=\int_{X}c_4(X)+\eta,
\end{equation}
where 
\begin{equation}\label{eqn:eta}
\eta=\sum_{x\in \Sing X}\left(1-\frac{1}{|G_x|}\right).
\end{equation}
\end{itemize}
Combining (\ref{0}), (\ref{1}), and (\ref{caraequa}), we can eliminate $\int_{X}c_4(X)$ and $\int_{X}\td_4(X)$ to obtain:
\begin{equation}
2h^{2,1}+h^{2,2}-8h^{1,1}=64+6(S_1-4S_0)+\eta.
\label{equa1}
\end{equation}
Similarly, by combining \eqref{0}, \eqref{2}, and \eqref{caraequa}, it yields that
\begin{equation}
2h^{2,1}+h^{2,2}-8h^{1,1}=64+3(S_2-6S_0)-2\eta.
\label{equa2}
\end{equation}
Then, (\ref{equa1}) and (\ref{equa2}) provide the following orbifold version (in dimension 4) of Salamon's famous relation \cite{Salamon} for hyper-K\"ahler varieties. See Appendix \ref{App:Salamon} for a generalization.
\begin{prop}[Orbifold Salamon relation]\label{RR}
Let $X$ be a primitively symplectic orbifold of dimension 4. We have:
\begin{equation}\label{eqn:SalamonOrb}
2h^{2,1}+h^{2,2}-8h^{1,1}=64+s,
\end{equation}
or equivalently,
\begin{equation*}
b_{4}+b_{3}-10b_{2}=46+s,
\end{equation*}
where
\begin{equation}\label{eqn:s}
 s:=6(S_1-4S_0)+\eta=3(S_2-6S_0)-2\eta=4S_{1}+S_{2}-22S_{0}
\end{equation} is a correction term determined by the singularities. In particular:
$$\eta=S_2-2S_1+2S_0=\sum_{i=0}^{4}(-1)^{i}S_{i}.$$
\end{prop}
\begin{rmk}
Proposition \ref{RR} shows that the knowledge of $h^{1,1}$, $h^{2,1}$ and the singularities is enough to compute the topological Euler characteristic and all the Betti numbers of a 4-dimensional primitively symplectic orbifold.
\end{rmk}

\subsection{Estimate of the contribution of singularities}\label{esti}
We turn to a more careful study of the quantity $s$ in the orbifold Salamon relation \eqref{eqn:SalamonOrb}, which is the local contribution of singularities. Using \eqref{eqn:Si}, \eqref{eqn:eta} and \eqref{eqn:s}, we can write $s=\sum_{x\in \operatorname{Sing} X} s_x$ with:
\begin{equation}
s_x=\frac{1}{|G_x|}\left[6\left(\sum_{\substack{g\in G_x\\g\neq \id}}\frac{\tr(\rho_{x,\Omega_X^{[1]}}(g))-4}{\det(\rho_{x, T_{X}}(g)-\id)}\right)+|G_x|-1\right],
\label{s}
\end{equation}
with $(U,V,G_x,\pi)$ a local uniformizing system around $x$. In particular, the action of $g\in G_x$ on $T_{V,0}$ is symplectic.
We can therefore write that  $g=\diag(\xi_1,\xi_2,\xi_1^{-1},\xi_2^{-1})$, with $\xi_j=e^{\frac{2i\pi k_{j}}{n_j}}$, $k_{j}, n_j\in\N$ for all $j\in\left\{1,2\right\}$. Hence:
\begin{align*}
\frac{\tr(\rho_{x,\Omega_X^1}(g))-4}{\det(g-\id)}&=\frac{2(\cos(\frac{2\pi k_{1}}{n_1})+\cos(\frac{2\pi k_{2}}{n_2}))-4}{4(1-\cos(\frac{2\pi k_{1}}{n_1}))(1-\cos(\frac{2\pi k_{2}}{n_2}))}\\
&=-\frac{1}{2(1-\cos(\frac{2\pi k_{1}}{n_1}))}-\frac{1}{2(1-\cos(\frac{2\pi k_{2}}{n_2}))}.
\end{align*}
So:
\begin{equation}
 \frac{\tr(\rho_{x,\Omega_X^1}(g))-4}{\det(g-\id)}\leq -\frac{1}{2}
 \label{trdet}
\end{equation}

Hence for any $x\in \Sing X$, we have
\begin{equation}\label{sx}
s_x\leq -2\left(\frac{|G_x|-1}{|G_x|}\right).
\end{equation}
In particular, $s_{x}\leq -1$ and the quantity $s$, which is an integer by \eqref{eqn:SalamonOrb}, is at most $-|\Sing X|$.

Using Proposition \ref{classification}, we can be more precise. The local fundamental group $G_x$ is a finite subgroup of $\Sp(2)$.
Hence, there exists a normal subgroup $G'$ of $G_{x}$ of index at most 2,  $H$ a finite subgroup of $\SU(2)$ and an automorphism $\theta$ of $H$
such that any element
$M\in G'$ has the form
$$M=\begin{pmatrix}
     A& 0\\
     0 & \theta(A)
    \end{pmatrix},$$
    with $A\in H$.
As we have noticed previously, if $A$ is a matrix of $\SU(2)$ of finite order, we have
$$\det(A-\id)=-\tr(A)+2.$$
Therefore
$$\sum_{\substack{g\in G'\\g\neq \id}}\frac{\tr(\rho_{x,\Omega_X^{[1]}}(g))-4}{\det(\rho_{x, T_{X}}(g)-\id)}=-\sum_{\substack{g\in G'\\g\neq \id}}\left(\frac{1}{2-\tr(A)}+\frac{1}{2-\tr(\theta(A))}\right),$$
   where on the right-hand side, we write a non-trivial element $g$ of $G'$ as  $\begin{pmatrix}
     A& 0\\
     0 & \theta(A)
    \end{pmatrix}$ for $A\in H$.  
Reordering the sum of the second term, we obtain the following equation:
\begin{equation}
 \sum_{\substack{g\in G'\\g\neq \id}}\frac{\tr(\rho_{x,\Omega_X^{[1]}}(g))-4}{\det(\rho_{x, T_{X}}(g)-\id)}=
-2\sum_{\substack{A\in H\\A\neq \id}}\frac{1}{2-\tr(A)}.
\label{sG'}
\end{equation}

\begin{ex}\label{scomputed}
We compute explicitly $s_x$ in the following cases.
\begin{itemize}
\item
$G_x=C_n$ is a cyclic group of order $n$. 

In this case, $G_x=G'$ and $H=\left<g_n\right>$,
with  $g_n=\diag(e^{\frac{2i\pi}{n}},e^{-\frac{2i\pi}{n}})$.
By (\ref{sG'}), we have:
\begin{equation}
\sum_{\substack{g\in C_n\\g\neq \id}}\frac{\tr(\rho_{x,\Omega_X^{[1]}}(g))-4}{\det(\rho_{x, T_{X}}(g)-\id)}=
-\sum_{k=1}^{n-1}\frac{1}{1-\cos(\frac{2k\pi}{n})}=-\frac{n^2-1}{6},
\label{scyclictr}
\end{equation}
where we used the identity
$$\sum_{k=1}^{n-1}\frac{1}{\sin^2(\frac{k\pi}{n})}=\frac{n^2-1}{3}.$$
As a result, 
\begin{equation}
 s_x(C_n)=-(n-1).
 \label{scyclic}
\end{equation}
\item
$G_x=\tilde{D}_n$ is a binary dihedral group of order $4n$.

In this case, $G_x=G'$ and $H=\tilde{D}_n$.
The binary dihedral group $\tilde{D}_n$ is generated by $B:=\begin{pmatrix} \xi & 0\\ 0 & \xi^{-1}\end{pmatrix}$ 
and $P=\begin{pmatrix} 0 & 1\\ -1 & 0 \end{pmatrix}$, with $\xi$ a $2n$-root of the unity.
The group $\tilde{D}_n$ can be partitioned into the disjoint union of the following two sets:
$$\left\{B,B^2,...,B^{2n}\right\}\ \text{and}\ \left\{BP,B^2P,...,B^{2n}P\right\}.$$
 Hence by (\ref{sG'}) and the fact that  $\tr(B^kP)=0$ for all $k\in\{1,...,2n\}$, we have that
$$
\sum_{\substack{g\in \tilde{D}_n\\g\neq \id}}\frac{\tr(\rho_{x,\Omega_X^{[1]}}(g))-4}{\det(\rho_{x, T_{X}}(g)-\id)}=
\left(\sum_{\substack{g\in C_{2n}\\g\neq \id}}\frac{\tr(\rho_{x,\Omega_X^{[1]}}(g))-4}{\det(\rho_{x, T_{X}}(g)-\id)}\right)-2n.
$$
Using (\ref{scyclictr}), we obtain
\begin{equation}
 \sum_{\substack{g\in \tilde{D}_n\\g\neq \id}}\frac{\tr(\rho_{x,\Omega_X^{[1]}}(g))-4}{\det(\rho_{x, T_{X}}(g)-\id)}=
 -\frac{4n^2+12n-1}{6}.
 \label{sdihedraltr}
\end{equation}
Therefore
\begin{equation}
 s_x(\tilde{D}_n)=-(n+2).
 \label{sdihedral}
\end{equation}
\end{itemize}
\end{ex}
\subsection{Orbifold Guan inequality}
In the smooth case, Guan \cite[Section 2]{Guan} has proved Theorem \ref{main} using two ingredients: the Hirzebruch--Riemann--Roch formula and the Verbitsky theorem (\cite[Theorem 1.5]{Verbitsky}). The orbifold extension of the former being explained in Section \ref{subsec:RRGB}, we give the generalization of the latter here.
\begin{prop}[{\cite[Proposition 5.16]{BakkerLehn}}]\label{sym}
Let $X$ be a primitively symplectic orbifold of dimension $2n$. Then the following map induced by the cup-product is injective for any $k\leq n$:
$$\Sym^k H^2(X,\C)\rightarrow H^{2k}(X,\C).$$
\end{prop}
\begin{rmk}
When $n=2$, we can also prove the previous proposition with an elementary method using the Fujiki relation and the fact that the Beauville--Bogomolov form is non-degenerate (see \cite[Section 3.4]{Lol}).
\end{rmk}
\begin{cor}\label{Guanineq}
Let $X$ be a primitively symplectic orbifold of dimension 4.
Then:
$$4h^{2,1}\leq -(h^{1,1})^2+15h^{1,1}+126+2s,$$
where $s$ is (the non-positive integer) defined in (\ref{s}).
In particular:
\begin{equation}
0\geq s\geq-91
\label{smax}
\end{equation}
\end{cor}
\begin{proof}
Proposition \ref{sym} provides the following inequality:
$$b_4\geq \frac{(b_2+1)b_2}{2},$$
which can be rewritten as
$$4+4h^{1,1}+2h^{2,2}\geq (3+h^{1,1})(2+h^{1,1}).$$
Combining this inequality with Proposition \ref{RR}, we obtain our result.
\end{proof}

\subsection{Proof of the main results}
\begin{proof}[Proof of Theorem \ref{main}]
Thanks to Corollary \ref{Guanineq}, we have:
$$0\leq -(h^{1,1})^2+15h^{1,1}+126+2s.$$
Or equivalently, $$-2s\leq  (21-h^{1,1})(h^{1,1}+6).$$
Because of (\ref{sx}), $s\leq 0$ and when $X$ is singular $s<0$. Then Theorem \ref{main} follows.
\end{proof}

\begin{proof}[Proof of Theorem \ref{main2}]
Statement $(i)$ follows from (\ref{sx}) and (\ref{smax}). 
Let us prove $(ii)$.

By Proposition \ref{classification}, there exist a normal subgroup $G'$ of $G_x$ of index at most 2,  $H$ a finite subgroup of $\SU(2)$ and an automorphism $\theta$ of $H$
such that any element
$M\in G'$ has the form
$$M=\begin{pmatrix}
     A& 0\\
     0 & \theta(A)
    \end{pmatrix},$$
    for some $A\in H$.
We only need the well-known classification of the finite subgroups of $\SU(2)$ to have a full description of all possible $G_x$.
The finite subgroups of $\SU(2)$ are the so-called \emph{Kleinian groups} corresponding to the A-D-E Dynkin diagrams:  the cyclic groups $C_n$, the binary dihedral groups $\tilde{D}_n$ and the three sporadic groups
$E_6$, $E_7$ and $E_8$. The biggest sporadic group $E_8$ has order 120. Let us check the maximal size of the groups of A-D types.

When $G_x=G'$, we already have computed in Example \ref{scomputed} that $s_x(C_n)=-(n-1)$ and $s_x(\tilde{D}_n)=-(n+2)$.

Now, we assume that $G_x/G'$ has order 2.
By (\ref{trdet}), we have:
\begin{align*}
s_x(G_x)&=\frac{1}{|G_x|}\left[6\left(\sum_{\substack{g\in G'\\g\neq \id}}\frac{\tr(\rho_{x,\Omega_X^{[1]}}(g))-4}{\det(\rho_{x, T_{X}}(g)-\id)}
+\sum_{\substack{g\in G'}}\frac{\tr(\rho_{x,\Omega_X^{[1]}}(T_{n,k}g))-4}{\det(\rho_{x, T_{X}}(T_{n,k}g)-\id)}\right)+|G_x|-1\right]\\
&\leq \frac{1}{|G_x|}\left[6\left(\sum_{\substack{g\in G'\\g\neq \id}}\frac{\tr(\rho_{x,\Omega_X^{[1]}}(g))-4}{\det(\rho_{x, T_{X}}(g)-\id)}
-\frac{|G'|}{2}\right)+|G_x|-1\right].
\end{align*}
Denote by $C_n^{[2]}$ (resp. $\tilde{D}_n^{[2]}$) a finite subgroup of $\Sp(2)$ such that $C_n$ (resp. $\tilde{D}_n$) is a normal subgroup of index 2. 
We have then by (\ref{scyclictr}) and (\ref{sdihedraltr}):
\begin{equation}
s_x(C_n^{[2]})\leq -\frac{n+1}{2}\ \text{and}\ s_x(\tilde{D}_n^{[2]})\leq -\frac{n+4}{2}.
\label{sindex2}
\end{equation}
However by (\ref{smax}), we know that $s_x(G)\geq -91$.
Hence, the biggest possible groups are the groups which have a binary dihedral group $\tilde{D}_{178}$ as normal subgroup of index 2.
These groups have order $8\times 178=1424.$
\end{proof}

\begin{rmk}
Using (\ref{smax}), we can be more precise about the maximal cardinality of each kind of groups. 
\begin{itemize}
 \item 
 If $G_x=C_n$ is a cyclic group of order $n$, then by (\ref{scyclic}), $n\leq 92$.
 \item
 If $G_x=\tilde{D}_n$ is a binary dihedral group of order $4n$, then by (\ref{sdihedral}), $n\leq 89$.
 \item
 If $G_x=C_n^{[2]}$ is a group with a cyclic group of order $n$ as normal subgroup of index 2, then by (\ref{sindex2}),
 $n\leq 181$.
 \item
 If  $G_x=\tilde{D}_n^{[2]}$ is a group with a binary dihedral group of order $4n$ as normal subgroup of index 2, then by (\ref{sindex2}),
 $n\leq 178$.
\end{itemize}
\end{rmk}
\begin{rmk}
Using Corollary \ref{Guanineq}, we can get sharper constraints on singularities for each fixed second Betti number. For example, if $b_{2}=22$ (\emph{resp.} 21, 20, \emph{etc.}), then there are at most 13 (\emph{resp.} 25, 36, \emph{etc.}) singular points.
\end{rmk}

\section{Hitchin--Sawon formula}\label{sec:HSformula}
We can try to improve Theorem \ref{main} using the same method as in \cite[Section 3]{Guan}. The method of Guan is based on an equation of Hitchin--Sawon \cite{Hitchin}. This section is just an attempt and is not needed in the rest of the paper. First, we recall the following generalized Fujiki formula.
\begin{lemme}[{\cite[Lemma 4.6]{Lol}}]\label{genefuji}
Let $X$ be a primitively symplectic orbifold of dimension $2n$. If $\beta\in H^{4p}(X,\C)$ is of type $(2p,2p)$ on all small deformations of $X$, then there exists a constant $N(\beta)$ depending on $\beta$ such that for all $\alpha\in H^{2}(X,\C)$, one has $\int_X\beta\cdot\alpha^{2(n-p)}=N(\beta) \left(\int_X\alpha^{2n}\right)^{\frac{n-p}{n}}$.
\end{lemme}
\begin{prop}\label{Hitchinorbi}
Let $X$ be a primitively symplectic orbifold of dimension $2n$. Then:
$$\frac{((2n)!)^{n-1}N(c_2)^n}{(24n(2n-2)!)^n}=\int_{X}\sqrt{\hat{A}},$$
where $\hat{A}$ is the $\hat{A}$-\emph{genus} defined by $\prod_{i=1}^{2n}\left(\frac{\sqrt{a_i}/2}{\sinh \sqrt{a_i}/2}\right)$, where $a_{i}$'s are the Chern roots of the tangent bundle of $X$.
\end{prop}
\begin{proof}
We adapt Hitchin--Sawon's proof \cite{Hitchin}. We will note that quotient singularities do not have any effect on Hitchin--Sawon formula.

We can consider $N(c_2)$ as defined in Lemma \ref{genefuji}. 
In the smooth case, the equation of Hitchin--Sawon provides an expression of $N(c_2)$ in terms of Pontryagin classes. The main tool used by Hitchin and Sawon are Rozansky--Witten invariants (see \cite[Section 2]{Hitchin}). 
These invariants are constructed from the curvature of the manifold and a trivalent graph with $2n$ vertices.

The tangent sheaf $T_X$ on $X$ can be defined as the unique reflexive sheaf such that $T_{X|X_{\reg}}$
is the usual holomorphic tangent sheaf on $X_{\reg}$. It is a locally V-free sheaf. 
We consider $g$ a K\"ahler metric on $T_X$.
As explained in Definition \ref{metric}, this provides a metric $g$ on $T_{X_{\reg}}$ such that for all local uniformizing system
$(U,V,G,\pi)$, $\pi^*(g_{|U_{\reg}})$ extends to a metric $g_V$ on $T_V$. 
Then the Riemannian curvature $K$ of $(X,g)$ is obtained on $X_{\reg}$ by the Riemannian curvature of $(X_{\reg},g)$ and on all local 
uniformizing systems by the Riemann curvature of $(V,g_V)$.
For the same reason as in the smooth case, we can associated to the curvature a class $[\Phi]\in H^1(X,\Sym^3 \Omega_X^1)$ (see \cite[Section 2]{Hitchin} or \cite{Witten}). 
From this class $[\Phi]$, the definition of the Rozansky--Witten invariants being purely algebraic, it can be generalized, word by word, to the case of primitively symplectic orbifold. 
We denote these invariants $b_{\Gamma}(X)$ for $\Gamma$ a trivalent graph with $2n$ vertices.

In the smooth case, it is well known that:
\begin{equation}
2c_2-c_1^2=\left[\frac{\tr K^2}{(2\pi)^2}\right],
\label{K2}
\end{equation}
where $K$ is the curvature of $X$. Because of our definition of Chern classes in Section \ref{reminders}, (\ref{K2}) is also true in the orbifold case.
In the symplectic case (\ref{K2}) gives:
$$2c_2=\left[\frac{\tr K^2}{(2\pi)^2}\right].$$
Then, using this expression for $c_2$ exactly as Hitchin and Sawon did in \cite[Section 3]{Hitchin}, we can provide an expression of $N(c_2)$ (\cite[equations (7) and (8)]{Hitchin}):
\begin{equation}
c_{\Theta}\int_X\omega^n\cdot\overline{\omega}^n=16\pi^2n\int_Xc_2\cdot\omega^{n-1}\cdot\overline{\omega}^{n-1},
\label{hitch8}
\end{equation}
where $\omega$ generated $H^{2,0}(X)$ and $c_{\Theta}$ can be express by:
\begin{equation}
b_{\Theta^n}(X)=\frac{n!}{(2\pi^2)^n}c_{\Theta}^n\vol(X),
\label{Hitch7}
\end{equation}
with $\vol(X)=\frac{\int_X\left(\omega+\overline{\omega}\right)^{2n}}{2^{2n}(2n)!}$ and $\Theta$ the trivalent graph with two vertices.
Equation (\ref{hitch8}) can be rewritten:
$$\frac{c_{\Theta}\int_X\left(\omega+\overline{\omega}\right)^{2n}}{{2n\choose n}}=\frac{16\pi^2n\int_Xc_2\cdot(\omega+\overline{\omega})^{n-1}}{{2(n-1)\choose n-1}}.$$
That is:
$$c_{\Theta}=\frac{32\pi^2(2n-1)N(c_2)}{\left[\int_X(\omega+\overline{\omega})^{2n}\right]^{1/n}}.$$
Then, with (\ref{Hitch7}), we obtain:
\begin{equation}
b_{\Theta^n}(X)=\frac{n!4^n(2n-1)^nN(c_2)^n}{(2n)!}.
\label{GuanHitch}
\end{equation}

In general, we can write:
$$s_{2m}=\left[\frac{\tr(K^{2m})}{(2\pi i)^{2m}}\right],$$
where $$\ch(T_X)=\sum_m\frac{s_{2m}}{(2m)!}.$$
Using these expressions and important results on graphs (see \cite[Section 5]{Hitchin}), Hitchin and Sawon provide an expression of $b_{\Theta^n}$ in terms of the Pontryagin classes. The results on graph are not affected by having singularities on $X$, hence, the same expression can be obtained in the symplectic orbifold case:
$$b_{\Theta^n}(X)=48^nn!\int_{X}\sqrt{\hat{A}}.$$
Combined with (\ref{GuanHitch}) this equation provides our proposition.
\end{proof}
\begin{lemme}\label{Guanorbi}
Let $X$ be a primitively symplectic orbifold of dimension 4, then:
$$3b_2N(c_2)^2\leq(b_2+2)c_2^2.$$
\end{lemme}
\begin{proof}
The proof of \cite[Lemma 3]{Guan} can be adapted in the case of primitively symplectic orbifolds. 
Indeed, it is a consequence of Lemma \ref{genefuji} and the Hodge--Riemann bilinear relation which have been generalized in \cite[Proposition 2.14]{Lol}. 
\end{proof}
\begin{cor}\label{hichinsawonineq}
Let $X$ be a primitively symplectic orbifold of dimension 4, then:
$$(b_2+1)b_3\leq 4b_2^2+2(S_1+20S_0-62)b_2+736+2(S_1-124S_0),$$
where $S_0$ and $S_1$ are defined in \eqref{eqn:Si}, introduced in Section \ref{secRR}.
\end{cor}
\begin{proof}
In our case, Proposition \ref{Hitchinorbi} provides:
$$
\frac{4!N(c_2)^2}{(24\times2\times2)^2}=\frac{1}{2}\td_4-\frac{c_2^2}{8\times 12^2}.
$$
That is:
$$N(c_2)^2=192\td_4-\frac{c_2^2}{3}.$$
So, using Lemma \ref{Guanorbi}:
$$
576b_2\td_4\leq2(b_2+1)c_2^2.
$$
In Section \ref{secRR}, we found expressions for $c_4$ and $\td_4$, so it is more convenient to replace $c_2^2$ by  $\frac{720\td_4+c_4}{3}$:
$$576b_2\td_4\leq2(b_2+1)\frac{720\td_4+c_4}{3}.$$
Then (\ref{1}) provides:
$$576b_2\td_4\leq2(b_2+1)(248\td_4+4b_2-b_3-8+2S_1).$$
This is:
$$(b_2+1)b_3\leq4b_2^2+(2S_1-40\td_4-4)b_2+248\td_4-8+2S_1.$$
Using (\ref{0}) to replace $\td_4$ by $3-S_0$, we obtain our result.
\end{proof}
\begin{ex}
We can apply Corollary \ref{hichinsawonineq} to orbifolds with singularities $\C^4/\pm\id$. 
It provides:
$$(b_2+1)b_3\leq4b_2^2+(N-124)b_2+736-8N,$$
where $N$ is the number of singular points.
\begin{itemize}
\item If $N=28$:
$$(b_2+1)b_3\leq 4(16-b_2)(8-b_2).$$
\item If $N=36$:
$$(b_2+1)b_3\leq 4(14-b_2)(8-b_2).$$
\end{itemize}
This corresponds exactly to the second Betti numbers of examples in \cite[Section 13]{Fujiki} (see also Section \ref{summ}).
\end{ex}
\begin{rmk}
Unfortunately, Corollary \ref{hichinsawonineq} is not restrictive enough to exclude more second Betti numbers. For instance, a primitively symplectic orbifold with second Betti number 22 and 3 isolated singularities of analytic type $(\C^4/g_3, 0)$ with $g_3=\diag(e^{\frac{2i\pi}{3}},e^{\frac{2i\pi}{3}},e^{-\frac{2i\pi}{3}},e^{-\frac{2i\pi}{3}})$ is not in contradiction with Corollary \ref{hichinsawonineq} and Proposition \ref{Guanineq}. 
To improve Theorem \ref{main}, we need some techniques to exclude some configurations of singularities.
\end{rmk}

\section{Examples of primitively symplectic orbifolds of dimension 4}\label{sec:Examples}
For each Betti number between 3 and 23, we provide an example of primitively symplectic orbifold when we know one. See \cite[Section 13]{Fujiki} for more examples; many additional examples could also be obtained by considering partial resolution in codimension 2 of quotients of $K3^{[2]}$-type and $\operatorname{Kum}_{2}$-type manifolds. We summarize all the numerical results in a table in Section \ref{summ}.

In this section, an  \emph{isolated cyclic quotient singularity of order $n$}  refers to the germ $(\C^k/C_n, 0)$, where  $C_n$ is an order-$n$ cyclic subgroup of $\Sp(k,\C)$ such that the origin is the only fixed point.

\subsection{A construction of Fujiki}
Most of the know examples of primitively symplectic orbifolds of dimension 4 was constructed by Fujiki in \cite[Section 13]{Fujiki}. 

We recall his construction. Let $H$ be a finite group of symplectic automorphisms on a K3 or an abelian surface $S$. First, we assume that $H$ is abelian. Let $\theta$ be an involution on $H$. The action of $H$ on $S\times S$ is given by $h\cdot (s,t)=(hs,\theta(h)t)$ for all $(h,s,t)\in H\times S \times S$.
Moreover, we define $G$ to be the group of automorphisms of $S\times S$ generated by $H$ and the involution $(s,t)\mapsto (t,s)$. The quotient $(S\times S)/G$ has isolated singularities as well as singularities in codimension 2. The singularities in codimension 2 can be resolved crepantly (see 
\cite{Birkar}) and we denote by $Y_{K3}(H)$ (resp. $Y_{T}(H)$) the primitively symplectic orbifold obtained when $S$ is a K3 surface (resp.~when $S$ is a complex torus of dimension 2). As explained in \cite[Section 13]{Fujiki}, the deformation class of $Y_{K3}(H)$ (resp. $Y_{T}(H)$) only depends on $H$.

When the group $H$ is non abelian, the situation is more complicated (the deformation class does not only depends on $H$) and Fujiki only provides 5 additional examples.

\subsection{$b_2=3$}
Mongardi \cite[Example 4.5.1 and Example 4.5.2]{MongT} constructed two manifolds $(X_1,\sigma_1)$ and $(X_2,\sigma_2)$ of $K3^{[2]}$-type endowed with symplectic automorphisms of order 11 such that
$$H^{2}(X_1,\Z)^{\sigma_1}\simeq \begin{pmatrix} 6 & 2 & 2 \\ 2 & 8 & -3\\ 2 & -3 & 8\end{pmatrix}\ \text{and}\  H^{2}(X_2,\Z)^{\sigma_2}\simeq\begin{pmatrix} 2 & 1 & 0 \\ 1 & 6 & 0\\ 0 & 0 & 22\end{pmatrix}.$$ 
We denote the quotients $M_{11}^i=X_i/\sigma_i$ with $i\in\left\{1,2\right\}$. The primitively symplectic orbifolds $M_{11}^1$ and $M_{11}^2$ both have second Betti number equal to 3 and have 5 isolated cyclic quotient singularities of order 11. Moreover their Beauville--Bogomolov forms were computed in \cite[Theorem 1.2]{Lol2}. The pair $(X_1, \sigma_1)$ was independently discovered in \cite{FuGlasgow}.

In general, the fourth Betti number of the quotient of a manifold of $K3^{[2]}$-type by an automorphism of prime order $p\neq 2,5$ can be computed using the Boissi\`ere--Nieper-Wisskirchen--Sarti invariants and the fact that:
$$\frac{H^4(X_i,\Z)}{\Sym^2 H^2(X_i,\Z)}=\left(\Z/2\Z\right)^{23}\oplus (\Z/5\Z).$$
See \cite[Section 2, Proposition 5.1, Proposition 6.6 and Lemma 6.14]{SmithTh} for more details. We obtain $b_4(M_{11}^i)=26$ for all $i\in\left\{1,2\right\}$.
\subsection{$b_2=5$}
Let $X$ be a manifold of $K3^{[2]}$-type endowed with a symplectic automorphism $\sigma$ of order 7. As  explained in \cite[Section 7.3]{MongT}, we always have $\rk H^2(X,\Z)^\sigma=5$ and $(X,\sigma)$ is standard (i.e.~deformation equivalent to a natural pair $(S^{[2]},\sigma^{[2]})$, where $S$ is a K3 surface and $\sigma^{[2]}$ induced by an automorphism $\sigma$ on $S$). We denote $M_7:=X/\sigma$, which is a primitively symplectic orbifold with second Betti number 5, and has 9 isolated cyclic quotient singularities of order 7. Moreover, its Beauville--Bogomolov form have been computed in \cite[Theorem 1.3]{Lol3}. 

In general, the Betti numbers of the quotient of a Hilbert scheme $S^{[m]}$ of $m$ points on a K3 surface $S$ by a natural automorphism of prime order can be computed using the Boissi\`ere--Nieper-Wisskirchen--Sarti invariants and the Qin--Wang integral basis of $H^*(S^{[m]},\Z)$ (\cite[Theorem 1.1 and Remark 5.6]{Wang}).
See \cite[Proof of Corollary 5.2]{Lol3} for more details. We can compute $b_4(M_7)=42$.
\subsection{$b_2=6$}\label{6}
We consider the complex torus $T=\C^2/\Lambda$, where $\Lambda=\left\langle (1,0),(i,0),(0,1),(0,i)\right\rangle$. The torus is given by the product of elliptic curves $T=E\times E$, with $E:=\C/\left\langle 1,i\right\rangle$.
Let $\sigma_4$ be the symplectic automorphism of order 4 on $T$ defined by the action on $\C^2$ given by the matrix:
$$\sigma_4=\begin{pmatrix} 0 & -1\\ 1 & 0\end{pmatrix}.$$  
We remark that 
\begin{equation}
H^2(T,\Z)^\nu=U\oplus \langle-2\rangle^{\oplus 2}.
\label{T4}
\end{equation}
Since $\sigma_4$ is a linear automorphism on $T$, it extends to an automorphism $\sigma_4^{[2]}$ on $K_2(T)$.
We want to consider a crepent resolution in codimension 2 of $K_2(T)/\sigma_4^{[2]}$. It can be obtain as follows.
We have $\sigma_4^2=-\id$. As explained in \cite[Section 1.2.1]{Tari}, the induced automorphism $(-\id)^{[2]}$ on $K_2(T)$ has 36 isolated fixed points and a fixed surface
$$\Sigma=\overline{\left\{\left.\xi\in K_2(T)\right|\ \Supp \xi=\left\{0,x,-x\right\}, x\in T\smallsetminus \left\{0\right\}\right\}}.$$
We can consider $r:\widetilde{K_2(T)}\rightarrow K_2(T)$ the blow-up of $K_2(T)$ in $\Sigma$. 
Then $\sigma_4^{[2]}$ extends to an automorphism $\widetilde{\sigma}_4^{[2]}$ on $\widetilde{K_2(T)}$ and $K_4':=\widetilde{K_2(T)}/\widetilde{\sigma}_4^{[2]}$ is a primitively symplectic orbifold, which is a crepent resolution in codimension 2 of $K_2(T)/\sigma_4^{[2]}$.
From (\ref{T4}), we deduce that:
\begin{equation} 
b_2(K_4')=4+1+1=6. 
\label{b2k4}
\end{equation}
Moreover, we can also describe $K_4'$ as a quotient of $K'$ (constructed in Section \ref{8}). The orbifold $K'$ is given by $\widetilde{K_2(T)}/(-\id)^{[2]}$. Then $\widetilde{\sigma}_4^{[2]}$ induced an involution on $K'$ that we denote by $\overline{\sigma}_4^{[2]}$ and $K_4'=K'/\overline{\sigma}_4^{[2]}$. By \cite[Proposition 8.23]{Lol4} $b_3(K')=0$, it follows that:
\begin{equation} 
b_3(K_4')=0.
\label{b3k4}
\end{equation}
Now, we are going to determine the singularities of $K_4'$, which are all isolated cyclic quotient singularities of order 2 or 4. Let us denote the number of such singular points by $a_2$ and $a_4$ respectively. It turns out to be quite technical to determine $a_2$ directly. However, order-4 cyclic quotient singular points  correspond to the singularities of $K'$ which are fixed by $\overline{\sigma}_4^{[2]}$, hence they correspond to the fixed points of $\sigma_4^{[2]}$ outside of $\Sigma$. After determining $a_4$, we will deduce $a_2$ with a numerical method. We have:
$$\Fix \sigma_4=\left\{\left.(a,a)\right|\ a\in E[2]\right\}.$$
That is, we have 4 isolated fixed points $(0,0)$, $(\frac{1}{2},\frac{1}{2})$, $(\frac{i}{2},\frac{i}{2})$ and $(\frac{1+i}{2},\frac{1+i}{2})$.
Hence $\sigma_4^{[2]}$ fixes only one point outside of the diagonal of $K_2(T)$ which is the scheme supported on $\left\{(\frac{1}{2},\frac{1}{2}), (\frac{i}{2},\frac{i}{2}), (\frac{1+i}{2},\frac{1+i}{2})\right\}$. This point is also outside of $\Sigma$.
Furthermore, $\sigma_4^{[2]}$ acts on 
$$Z_0:=\left\{\left.\xi\in K_2(T)\right|\ \Supp \xi=\left\{(0,0)\right\}\right\}.$$ 
As explained in \cite[Section 1.2.1]{Tari}, the fixed locus of $(\sigma_4^{[2]})^2=(-\id)^{[2]}$ on $Z_0$ is given by the vertex $x$ and a line $\ell$ which is contained in $\Sigma$; the vertex $x$ is out of $\Sigma$. Necessarily, the vertex $x$ will also be fixed by $\sigma_4^{[2]}$. 
We conclude that $\sigma_4^{[2]}$ only fixes 2 points outside of $\Sigma$ and so:
\begin{equation}
a_4=2.
\label{N3}
\end{equation}
Now, we are going to deduce the number $a_2$  by considering the double cover:
$$\pi: K'\rightarrow K'/\overline{\sigma}_4^{[2]}=K_4'.$$
Since $\pi$ has only isolated ramification points, the Hurwitz formula can be used in this framework and provides:
\begin{equation}
\chi(K')=2\chi(K_4')-R,
\label{Hurwitz}
\end{equation}
where $R$ is the number of ramification points.
An order-2 cyclic quotient singular point  in $K_4'$ can arise in two different ways. First, it can be the image by $\pi$ of a non-ramified singular point of $K'$,  or second, it can be the image by $\pi$ of a ramified smooth point of $K'$. Since the $\pi$-ramified singular points of $K'$ provide cyclic quotient singularities of order 4 in $K_4'$, we obtain the following formula:
$$a_2=\frac{\#\Sing(K')-a_{4}}{2}+R-a_4.$$
We have seen that $K'$ has 36 isolated singularities, hence by (\ref{N3}):
\begin{equation}
a_2=15+R.
\label{N1}
\end{equation}
By \cite[Proposition 8.23]{Lol4}, $\chi(K')=108$. Hence (\ref{Hurwitz}), (\ref{b3k4}) and Proposition \ref{RR} provide:
$$108=2(48+12b_2(K_4')+s)-R.$$
Applying (\ref{N1}), (\ref{b2k4}) and (\ref{scyclic}), we obtain:
$$108=2(48+12\times6-3\times 2-a_2)-a_2+15.$$
One deduce that  $a_2=45$.
\subsection{$b_2=7$}
The generalized Kummer fourfold \cite{Beauville}.

For the sake of completeness, we also provide a singular example with $b_2=7$.
We consider the same complex torus as previously $T=\C^2/\Lambda$, where $\Lambda=\left\langle (1,0),(i,0),(0,1),(0,i)\right\rangle$. The torus is given by the product of elliptic curves $T=E\times E$, with $E:=\C/\left\langle 1,i\right\rangle$.
Let $\nu$ be the symplectic automorphism of order 3 on $T$ defined by the action on $\C^2$ given by the matrix:
$$\nu=\begin{pmatrix} 0 & -1\\ 1 & -1\end{pmatrix}.$$  
We remark that: 
\begin{equation}
H^2(T,\Z)^\nu=U\oplus A_2.
\label{b27}
\end{equation}
Since $\nu$ is a linear automorphism on $T$, it extends to an automorphism on $K_2(T)$.
Furthermore, $\nu$ verifies the following relation: 
\begin{equation}
\id+\nu+\nu^2=0.
\label{eta}
\end{equation}
We denote $\nu^{[2]}$ the automorphism induced by $\nu$ on $K_2(T)$.
The relation (\ref{eta}) shows that $\nu^{[2]}$ admits one fixed surface $\Sigma$ which induces a surface of singularities in the quotient $K_2(T)/\nu^{[2]}$.
This surface is isomorphic to the K3 surface obtained after resolving the singularities of $T/\nu$.
Because of (\ref{eta}), the fixed points of $\nu$ are in $T[3]$, there are 9 points of the form $(a,2a)$, where $a\in E[3]$. Let $x_1,...,x_9$ be these 9 fixed points. It induces $\frac{9\times8\times1}{6}=12$ fixed points of $\nu^{[2]}$ of the form 
$\left\{x_i,x_j,-x_i-x_j\right\}$, with $x_i\neq x_j$. 
Let $$Z_\tau:=\left\{\left.\xi\in K_2(T)\right|\ \Supp \xi=\left\{\tau\right\}\right\},$$ for all $\tau=(a,2a)$, with $a\in E[3]$.
As explained in \cite[Section 4]{Hassett}:
$$Z_\tau\simeq \mathbb{P}(1,1,3).$$
Hence,
the action of $\nu^{[2]}$ on $Z_\tau$ fixes 2 lines which intersect in the vertex of $Z_\tau$. These two lines are necessarily included in the surface $\Sigma$. Hence there is no additional isolated fixed point in $Z_\tau$; that is $\nu^{[2]}$ has 12 isolated fixed points.

Let $\overline{\Sigma}$ be the image of $\Sigma$ in $K_2(T)/\nu^{[2]}$. Let $x\in \overline{\Sigma}$. The singularity $(K_2(T)/\nu^{[2]},x)$ is analytically equivalent to $(\C^2\times (\C^2/g_3),0)$, with 
$g_3=\diag(\xi_3,\xi_3^{-1})$ and $\xi_3=e^{\frac{2i\pi}{3}}$. The singularity $(\C^2/g_3,0)$ is of type $A_2$ and can be resolved crepently by a blow-up and this resolution has two exceptional lines which intersect in one point. This shows that the singular surface $\overline{\Sigma}$ can be resolved crepently and the exceptional locus is the union of two irreducible divisors. Let $K_3'\rightarrow K_2(T)/\nu^{[2]}$ be such a resolution. The orbifold $K_3'$ is primitively symplectic with 12 isolated cyclic quotient singularities of order 3. Moreover by (\ref{b27}), we have:
$$b_2(K'_3)=4+1+2=7.$$

Furthermore, because of the action of $\nu$ on $\Lambda$ and \cite[Corollary 6.3]{Lol4}, the third Betti number of $K_2(T)/\nu^{[2]}$ is trivial. Since $\Sigma$ is a simply connected surface, $b_3(K_3')=0$. 
\subsection{$b_2=8$}\label{8}
This example has already been introduced in Section \ref{6} as an intermediate consequence of computation. In full generality, it can be constructed as follows.
We consider a symplectic involution $\iota$ on $X$ a manifold of $\operatorname{Kum}_2$-type. As it is explained in \cite[Theorem 7.5]{Lol4}, $X/\iota$ has a surface of singularities and 36 isolated fixed points. By resolving the surface of singularities, we obtain a primitively symplectic orbifold that we denote $K'$ and which has $b_2=8$.
Moreover, its Beauville--Bogomolov form has been computed in \cite[Theorem 1.1]{Lol4} and its Betti numbers in \cite[Proposition 8.23]{Lol4}. It has been proved that $K'$ is irreducible symplectic in \cite[Proposition 3.8]{Lol}.
\subsection{$b_2=9$}
 We describe Fujiki's example with second Betti number 9.
 Let $T$ be a complex torus which admits a symplectic binary dihedral 
 linear automorphism group $\tilde{D}_3$ of order 12.
 For instance, we consider $T=E_{\xi_6}\times E_{\xi_6}$, where $E_{\xi_6}=\C/\left<1,\xi_6\right>$ and $\xi_6=e^{\frac{i\pi}{3}}$.
 Then $\tilde{D}_3$ is generated by the linear automorphisms:
 $$\begin{pmatrix}\xi_6 & 0\\ 0 & \xi_6^{-1}\end{pmatrix}\ \text{and}\ \begin{pmatrix} 0 & 1\\ -1 & 0\end{pmatrix}.$$
 Let $N$ be the center of $\tilde{D}_3$ which is generated by $-\id$. We consider the K3 surface $S$ obtained as a resolution of $T/N$. 
 The group $H=\tilde{D}_3/N$ is isomorphic to the dihedral group of order 6, denoted by $D_3$. There is a natural lifting of a symplectic action of $H$ on $S$. 
 Then, as in the abelian case, we form the automorphisms group $G$ on $S\times S$ generated by $H$ acting diagonally and the involution $(s,t)\mapsto (t,s)$ with $\theta=\id$. 
 Fujiki resolves the singularities in codimension 2 of $(S\times S)/G$ and shows in \cite[Section 13]{Fujiki} that we obtain a primitively symplectic orbifold $Y_{K3}(D_3)$ with second Betti number 9.
 
 Since $\theta=\id$, the resolution in codimension 2 considered by Fujiki in \cite[Section 7]{Fujiki} corresponds to 
 $S^{[2]}/H\rightarrow (S\times S)/G$,
 where 
 $S^{[2]}$ is the Hilbert scheme of 2 points on $S$ and $H$ the induced automorphisms group. That is $Y_{K3}(D_3)=S^{[2]}/H$.

To determine the singularities of $Y_{K3}(D_3)$, we start by computing the singularities of $S/H$.
As there is no fixed point of $S$ by the entire group $H$, by classification of the finite subgroups of $\SL(2,\C)$, we know that the singularities of $S/H$ can only be of type $A_1$ or $A_2$.
We denote by $N_2$ (resp. $N_1$) the number of singularities of $S/H$ of type $A_2$ (resp. $A_1$).
Then the integers $N_2$ and $N_1$ can be computed by Riemann--Roch (a direct computation is also possible, but is more technical).
We apply Theorem \ref{thm:RR} to $S/H$ and the V-bundles $\mathcal{O}_{S/H}$, $\Omega_{S/H}^{[1]}$.
We obtain respectively:
\begin{equation}
2=\frac{c_2}{12}+\frac{2N_2}{9}+\frac{N_1}{8},
\label{eqK31}
\end{equation}
and
\begin{equation}
-6=-\frac{5c_2}{6}-\frac{2N_2}{9}-\frac{N_1}{4}.
\label{eqK32}
\end{equation}
Finally, Theorem \ref{cara} provides:
\begin{equation}
 10=c_2+\frac{2N_2}{3}+\frac{N_1}{2}.
 \label{eqK33}
\end{equation}

Combining (\ref{eqK31}), (\ref{eqK32}) and (\ref{eqK33}), we obtain $N_1=0$ and $N_2=7$.
Then, we can deduce the singularities of $S^{[2]}/H$. There are $\frac{7\times 6}{2}=21$ singular points of the form $(a,b)$ with 
$a\neq b\in\Sing S/H$ and $2\times 7=14$ singular points on the diagonal. So $Y_{K3}(D_3)$ has 35 isolated cyclic quotient singularities of order 3.

\subsection{$b_2=10$}
We have examples of Fujiki, for instance for $H=\Z/4\Z$ and $S$ a K3 surface. 
\subsection{$b_2=11$}
Let $X$ be a manifold of $K3^{[2]}$-type endowed with a symplectic automorphism $\sigma$ of order 3. As  explained in \cite[Section 7.3]{MongT}, there are two possibilities: $\sigma$ has 27 isolated fixed points or $\sigma$ has an abelian fixed surface. When $\Fix \sigma=\left\{27 \text{ points}\right\}$,
 we always have $\rk H^2(X,\Z)^\sigma=11$. We denote $M_3:=X/\sigma$ which is a primitively symplectic orbifold with second Betti number 11 and 27 isolated cyclic quotient singularities of order 3. Moreover its Beauville--Bogomolov form has been computed in \cite[Theorem 1.3]{Lol2}.
\subsection{$b_2=14$}
We have examples of Fujiki, for instance for $H=(\Z/2\Z)^2$ and $S$ a K3 surface. 
\subsection{$b_2=16$}
We consider a symplectic involution $\iota$ on $X$ a manifold of $K3^{[2]}$-type. As it is explained in \cite{Mong0}, $X/\iota$ has a surface of singularities and 28 isolated fixed points. If we resolve the surface of singularities we obtain a primitively symplectic orbifold that we denote $M'$ and which has $b_2=16$.
Moreover, its Beauville--Bogomolov form has been computed in \cite[Theorem 2.5]{Lol5} and its fourth Betti number in \cite[Proposition 2.40]{Lol5}. It has been proved that $M'$ is symplectic irreducible in \cite[Proposition 3.8]{Lol}.

\subsection{$b_{2}=23$} The Hilbert square of K3 surface \cite{Beauville}.

\subsection{Summary}\label{summ}
\[
\begin{tabular}{|c|c|c|c|c|c|c|}
\hline
$b_2$ & $X$ & $b_3$ & $b_4$ & singularities& B--B form& irreducible\\
\hline
3 & $M_{11}^{1}$; $M_{11}^{2}$ & 0 & 26& $a_{11}=5$ & $\begin{pmatrix} 2 & -1 & 3 \\ -1 & 8 & -1\\ 3 & -1 & 6\end{pmatrix}$; $\begin{pmatrix} 2 & 1 & 0 \\ 1 & 6 & 0\\ 0 & 0 & 2\end{pmatrix}$& no\\
\hline
4 &? & ?& ?&? &? &? \\
 \hline
 5 & $M_7$ & 0 & 42 & $a_7=9$ & $U\oplus\begin{pmatrix} 4 & -3\\ -3 & 4 \end{pmatrix}\oplus(-14)$  & no\\
 \hline
 6 & $K_4'$ & 0 & 55 & $a_2=45$, $a_4=2$ &? & no\\ 
 \hline
 7 & $K_2(A)$ & 8 & 108 & smooth & $U^3\oplus (-6)$ & yes\\
 \hline
 7 & $K_3'$ & 0 & 92 & $a_3=12$  &? & ?\\ 
 \hline
 8 & $K'$ & 0 & 90 & $a_2=36$ & $U(3)^3\oplus\begin{pmatrix} -5 & -4\\ -4 & -5 \end{pmatrix}$ & yes\\
 \hline
 9 & $Y_{K3}(D_3)$  & 0 & 66 & $a_3=35$ & ? & no\\
 \hline
 10 & $Y_{K3}(\Z/4\Z)$ & 0 & 118 & $a_2=10$, $a_4=6$ & ? & ?\\
 \hline
 11 & $M_3$ & 0 & 102 & $a_3=27$& $U(3)\oplus U^2\oplus A_2^2\oplus(-6)$ & no\\
 \hline
 12-13 &? & ?& ?&? &? &? \\
 \hline
 14 &$Y_{K3}\left((\Z/2\Z)^2\right)$ & 0 & 150 & $a_2=36$ & ? & ?\\
  \hline
15 &? & ?& ?&? &? &? \\
  \hline
 16 & $M'$ & 0 & 178 & $a_2=28$ & $U(2)^3\oplus E_8(-1)\oplus (-2)^2$ & yes\\
 \hline
 17-22 &? & ?& ?&? &? &? \\
  \hline
  23 & $K3^{[2]}$ & 0 & 276 & smooth & $U^3\oplus E_8(-1)^2\oplus (-2)$ & yes\\
   \hline 
\end{tabular} 
 \]
 Here, we denote by $a_k$ the number of isolated cyclic quotient singularities of order $k$.
 
  The singularities of the Fujiki examples are described in \cite[Section 13, table 1]{Fujiki}. We can then compute their fourth Betti numbers using Proposition \ref{RR} and Example \ref{scomputed}. Same thing for all the orbifolds with known singularities and $b_3$.

\appendix
\section{Salamon's relation for orbifolds}\label{App:Salamon}
 In this Appendix, we extend some of Salamon's results \cite{Salamon} to the setting of complex orbifolds. Let us first recall some notation. Given a complex orbifold $X$, we denote by $\Omega_{X}^{[p]}:=(\Omega_{X}^{p})^{\vee\vee}\cong \iota_{*}(\Omega_{X_{\operatorname{reg}}}^{p})$ the sheaf of reflexive  $p$-forms, where $\iota: X_{\operatorname{reg}}\to X$ is the inclusion of the smooth part. 
 The main result of Appendix is the following generalization of \cite[Corollary 3.4]{Salamon}.

\begin{thm}\label{thm:OrbSalamon}
Let $X$ be a compact complex orbifold of dimension $n$. 
Then 
\begin{equation}\label{eqn:SalamonOrbifold}
\int_{X} c_{1}c_{n-1}=\sum_{p=0}^{n} (-1)^{p}(6p^{2}-\frac{1}{2}n(3n+1))\chi^{\orb}_{p},
\end{equation}
where $\chi_{p}^{\orb}:=\int_{X}\ch(\Omega_{X}^{[p]})\cdot \td(X)$.
\end{thm}
\begin{proof}
Using the fact that $\Omega_{X}^{[p]}\cong {\Omega_{X}^{[n-p]}}^{\vee}\otimes \omega_{X}$, we easily see from the definition that 
\begin{equation}\label{eqn:chi-symmetry}
 \chi_{p}^{\orb}=(-1)^{n}\chi_{n-p}^{\orb}.
\end{equation}
We define the orbifold $\chi_{y}$-genus of $X$ to be the following polynomial in $t$:
\begin{equation}
\chi_{\orb}(t):=\sum_{p=0}^{n}\chi_{p}^{\orb}t^{p}=(-1)^{n}\sum_{p=0}^{n}\chi^{\orb}_{n-p}t^{p}.
\end{equation} 
Let $c(T_{X})=\prod_{i=1}^{n}(1+x_{i})$ be the formal factorization of the Chern class of $X$, where $x_{1}, \dots, x_{n}$ are the Chern roots. Following \cite{Salamon}, we define a new polynomial in $t$ with coefficients characteristic classes of $X$:
$$K(t):=\prod_{i=1}^{n}(x_{i}+t\cdot\frac{x_{i}}{1-e^{-x_{i}}}).$$
Denote by $K_{k}\in H^{*}(X, \Q)$ the coefficient of $t^{k}$.
Clearly, $K_{0}=c_{n}$. More generally, as shown in \cite[Proposition 3.2]{Salamon}, $K_{k}$ has the nice property that $K_{k}-c_{n-k}$ is in the ideal generated by $c_{n-k+1}, \dots, c_{n}$ for any $k\geq 1$. For example,
\begin{equation}\label{eqn:K1}
K_{1}=c_{n-1}+\frac{1}{2}nc_{n};
\end{equation} 
\begin{equation}\label{eqn:K2}
K_{2}=c_{n-2}+\frac{1}{2}(n-1)c_{n-1}+\frac{1}{24}(2c_{1}c_{n-1}+n(3n-5)c_{n}).
\end{equation} 
It is easy to relate $K(t)$ and $\chi_{\orb}(t)$:
\begin{eqnarray*}
\int_{X}K(t)&=&\int_{X}\td(X)\cdot \prod_{i=1}^{n}(t+1-e^{-x_{i}})\\
&=& \int_{X}\td(X)\cdot \sum_{p=0}^{n}(t+1)^{p}(-1)^{n-p}\sum_{1\leq j_{1}<\dots<j_{n-p}\leq n}e^{-x_{j_{1}}-\dots -x_{j_{n-p}}}\\
&=&\int_{X}\td(X)\cdot \sum_{p=0}^{n}(t+1)^{p}(-1)^{n-p}\ch(\Omega_{X}^{[n-p]})\\
&=&(-1)^{n}\sum_{p=0}^{n}(-1-t)^{p}\chi^{\orb}_{n-p}\\
&=&\chi_{\orb}(-1-t).
\end{eqnarray*}
Therefore for any $k\geq 0$,
\begin{equation}\label{eqn:Kchi}
 \int_{X}K_{k}=\frac{1}{k!}\int_{X} K^{(k)}(0)=\frac{(-1)^{k}}{k!}\chi_{\orb}^{(k)}(-1).
\end{equation}
Taking $k=0$ in \eqref{eqn:Kchi}, we obtain 
\begin{equation}\label{eqn:k=0}
 \int_{X}c_{n}=\sum_{p=0}^{n}(-1)^{p}\chi_{p}^{\orb}.
\end{equation}
(This is essentially Blache's Gauss--Bonnet Theorem \ref{cara}.)\\
Taking $k=1$ in \eqref{eqn:Kchi}, combined with \eqref{eqn:K1}, we find 
\begin{equation}\label{eqn:k=1}
\frac{n}{2}\int_{X}c_{n}=\sum_{p=0}^{n}(-1)^{p}p\chi_{p}^{\orb}.
\end{equation} 
(This is equivalent to \eqref{eqn:k=0} by taking into account the symmetry \eqref{eqn:chi-symmetry}.)\\
Taking $k=2$ in \eqref{eqn:Kchi}, combined with \eqref{eqn:K2}, it yields that 
\begin{equation}\label{eqn:k=2}
\frac{1}{12}\int_{X} c_{1}c_{n-1}+\frac{n(3n-5)}{24}\int_{X}c_{n}=\sum_{p=0}^{n}\frac{(-1)^{p}}{2}p(p-1)\chi_{p}^{\orb}.
\end{equation}
We deduce \eqref{eqn:SalamonOrbifold} by combining \eqref{eqn:k=0}, \eqref{eqn:k=1}, and \eqref{eqn:k=2}.
\end{proof}

\begin{rmk}
We are mainly interested in the Hodge numbers $h^{p,q}:=\dim H^{q}(X, \Omega_{X}^{[p]})$ and their alternating sum $$\chi_{p}:=\chi(X, \Omega_{X}^{[p]}):=\sum_{q=0}^{n}(-1)^{q}h^{p,q},$$ rather than $\chi_{p}^{\orb}$. Therefore, in practice, Theorem \ref{thm:OrbSalamon} is often combined with orbifold Riemann--Roch formula relating these two quantities: when $X$ has only isolated singularities, Blache's Riemann--Roch theorem \ref{thm:RR} implies that for any integer $p\geq 0$,
  $\chi^{\orb}_{p}=\chi_{p}-S_{p}$ with 
\begin{equation}\label{eqn:SpIso}
S_p:=\sum_{x\in\Sing X} \frac{1}{|G_x|}\sum_{g\neq\id}\frac{\tr(\rho_{x,\Omega_X^{[p]}}(g))}{\det(\id-\rho_{x, T_{X}}(g))}
\end{equation}
being the correction term determined by the singularities of $X$,
where for any singular point $x\in X$, $G_{x}$ is the local fundamental group and for any locally V-free bundle $F$, $\rho_{x, F}$ is the associated representation of $G_{x}$; see Section \ref{reminders}.
In the general case where the singularities are not necessarily isolated, we can use Kawasaki's Riemann--Roch formula \cite{Kawasaki}, by replacing in the above definition of $S_{p}$ by 
\begin{equation}\label{eqn:SpGeneral}
 S_{p}:=\sum_{\mu} \frac{1}{m_{\mu}}\int_{X_{\mu}} \ch\left(\frac{\tr\Omega_{X}^{[p]}}{\tr\wedge^{\bullet} N^{\vee}_{\mu}}\right)\cdot \td(T_{X_{\mu}}),
\end{equation}
where $\mu$ runs over all connected components of the inertia orbifold $IX$ except the component $X$. We refer to \cite{Kawasaki, Tonita} for the explanation of the notation as well as more details.
\end{rmk}

\begin{rmk}\label{rmk:SpSymmetry}
Given an integer $p$, from \eqref{eqn:chi-symmetry} and the fact $\chi_{p}=(-1)^{n}\chi_{n-p}$, we see that $$S_{p}=(-1)^{n}S_{n-p}.$$ In fact, if $X$ has only isolated singularities, it is a simple exercise in linear algebra to see more precisely that for any $x\in \Sing X$, $$S_{p, x}=(-1)^{n}S_{n-p, x},$$
where $$S_{p, x}:=\frac{1}{|G_x|}\sum_{g\neq\id}\frac{\tr(\rho_{x,\Omega_X^{[p]}}(g))}{\det(\id-\rho_{x, T_{X}}(g))}.$$
\end{rmk}

When specializing Theorem \ref{thm:OrbSalamon} to the symplectic case, we get the following orbifold analogue of \cite[Theorem 4.1]{Salamon}.

\begin{thm}\label{thm:SalamonHK}
Let $X$ be a compact K\"ahler orbifold of even complex dimension $n=2m$ such that the Hodge numbers satisfy the ``mirror symmetry'' $h^{p, q}=h^{n-p, q}$ for any $p, q\geq 0$. Then 
\begin{equation}\label{eqn:SalamonMS}
\int_{X} c_{1}c_{n-1}=\sum_{k=0}^{2n}(-1)^{k}(3k^{2}-\frac{1}{2}n(6n+1))b_{k}-\sum_{p=0}^{n}(-1)^{p}(6p^{2}-\frac{1}{2}n(3n+1))S_{p},
\end{equation}
where $S_{p}$ is defined in \eqref{eqn:SpIso} when $X$ has only isolated singularities and in \eqref{eqn:SpGeneral} in general.
\end{thm}
The most important case we have in mind is when $X$ is a primitively symplectic orbifold (Definition \ref{def:PSO}), hence the left-hand side of \eqref{eqn:SalamonMS} vanishes and we get a linear relation among the Betti numbers and contributions of singularities.
\begin{proof}
We keep the same notation as above. Since $\chi_{p}^{\orb}=\chi_{p}-S_{p}$, the equations \eqref{eqn:SalamonOrbifold} and \eqref{eqn:k=0} imply that 
\begin{equation}\label{eqn:phi20}
\sum_{p, q=0}^{n}(-1)^{p+q}p^{2}h^{p,q}-\sum_{p}(-1)^{p}p^{2}S_{p}=\sum_{p=0}^{n} (-1)^{p}p^{2} \chi_{p}^{\orb}=\frac{1}{6}\int_{X} c_{1}c_{n-1}+\frac{1}{12}n(3n+1)\int_{X} c_{n}.
\end{equation}
Using the Hodge symmetry $h^{p,q}=h^{q,p}$, we have
\begin{equation}\label{eqn:phi02}
\sum_{p, q=0}^{n}(-1)^{p+q}q^{2}h^{p,q}-\sum_{p}(-1)^{p}p^{2}S_{p}=\frac{1}{6}\int_{X} c_{1}c_{n-1}+\frac{1}{12}n(3n+1)\int_{X} c_{n}.
\end{equation}
The mirror symmetry relation $h^{p,q}=h^{p, n-q}$ implies that (remember that $n$ is even)
\begin{equation*}
 \sum_{p, q=0}^{n}(-1)^{p+q}pqh^{p,q}=\sum_{p, q=0}^{n}(-1)^{p+q}p(n-q)h^{p,q}.
\end{equation*}
Hence 
\begin{equation*}
 2\sum_{p, q=0}^{n}(-1)^{p+q}pqh^{p,q}=\sum_{p, q=0}^{n}(-1)^{p+q}nph^{p,q}=n\sum_{p=0}^{n}(-1)^{p}p\chi_{p}.
\end{equation*}
Writing $\chi_{p}=\chi_{p}^{\orb}+S_{p}$, together with \eqref{eqn:k=1}, we obtain
\begin{equation}\label{eqn:phi11}
2\sum_{p, q=0}^{n}(-1)^{p+q}pqh^{p,q}=n\sum_{p=0}^{n}(-1)^{p}p\chi_{p}^{\orb}+n\sum_{p=0}^{n}(-1)^{p}pS_{p}=\frac{n^{2}}{2}\int_{X}c_{n}+n\sum_{p=0}^{n}(-1)^{p}pS_{p}.
\end{equation}
Combining \eqref{eqn:phi20}, \eqref{eqn:phi02}, and \eqref{eqn:phi11}, we find
$$\sum_{p, q=0}^{n}(-1)^{p+q}(p+q)^{2}h^{p,q}-\sum_{p=0}^{n}(-1)^{p}p(2p+n)S_{p}=\frac{1}{3}\int_{X}c_{1}c_{n-1}+(n^{2}+\frac{1}{6}n)\int_{X}c_{n},$$
which is nothing else but
\begin{equation}\label{eqn:Betti}
\sum_{k=0}^{2n}(-1)^{k}k^{2}b_{k}-\sum_{p=0}^{n}(-1)^{p}p(2p+n)S_{p}=\frac{1}{3}\int_{X}c_{1}c_{n-1}+(n^{2}+\frac{1}{6}n)\int_{X}c_{n},
\end{equation}
On the other hand, \eqref{eqn:k=0} says that
\begin{equation}\label{eqn:cnBetti}
 \int_{X}c_{n}=\sum_{k=0}^{2n}(-1)^{k}b_{k}-\sum_{p=0}^{n}(-1)^{p}S_{p}.
\end{equation}
Putting \eqref{eqn:Betti} and \eqref{eqn:cnBetti} together, we obtain the following formula,
\begin{equation}
\int_{X} c_{1}c_{n-1}=\sum_{k=0}^{2n}(-1)^{k}(3k^{2}-\frac{1}{2}n(6n+1))b_{k}-\sum_{p=0}^{n}(-1)^{p}(6p^{2}+3np-3n^{2}-\frac{1}{2}n)S_{p}.
\end{equation}
 The desired formula then can be deduced using the fact that $S_{p}=S_{n-p}$, see Remark \ref{rmk:SpSymmetry}.
\end{proof}

We recover Proposition \ref{RR} as a special case of Theorem \ref{thm:SalamonHK}.

\bibliographystyle{amssort}

\noindent
Lie \textsc{Fu}

\noindent
Institut Camille Jordan, Universit\'e Claude Bernard Lyon 1, Lyon (France)

\noindent
IMAPP, Radboud Universiteit, Nijmegen (Netherlands)

\noindent
{\tt fu@math.univ-lyon1.fr}\\

\noindent
Gr\'egoire \textsc{Menet}

\noindent
Institut Fourier

\noindent 
100 rue des Math\'ematiques, Gi\`eres (France),

\noindent
{\tt gregoire.menet@univ-grenoble-alpes.fr}

\end{document}